\documentclass[11pt]{amsart}

\usepackage{amscd}
\usepackage{tensor}
\usepackage{comment}
\usepackage{color}
\usepackage{amssymb,a4wide}
\usepackage{amsthm}
\usepackage{graphicx}
\usepackage[T1]{fontenc}

\newcommand{\weg}[1]{}
\newcommand{\bq}{\begin{equation}}
\newcommand{\eq}{\end{equation}}
\renewcommand{\d}{\mathrm{d}}

\newcommand{\R}{\mathbb{R}}

\newcommand{\Id}{\mathrm{Id}}

\newcommand{\A}{\mathcal{A}}

\newtheorem{thm}{Theorem}
\newtheorem{lem}[thm]{Lemma}
\newtheorem{cor}[thm]{Corollary}
\newtheorem{prop}[thm]{Proposition}
\theoremstyle{definition}
\newtheorem{rem}{Remark}
\newtheorem{ex}{Example}

\newtheorem*{defn}{Definition}

\title[The degree of mobility of Einstein metrics]
{The degree of mobility of Einstein metrics} 
  \author{Vladimir S. Matveev and Stefan Rosemann}

\address{Institute of Mathematics, Friedrich-Schiller-Universit\"at Jena, 07737 Germany.}
\email{vladimir.matveev@uni-jena.de,   stefan.rosemann@uni-jena.de}
\email{}

\begin{document}

\begin{abstract}
Two pseudo-Riemannian metrics are called projectively equivalent if their unparametrized 
geodesics coincide. The degree of mobility of a metric is the dimension of the space of metrics 
that are projectively equivalent to it. We give a complete  list of  possible values  for  the degree of 
mobility of Riemannian and Lorentzian Einstein metrics on simply connected manifolds, and  describe all
possible dimensions of the space of essential projective vector fields.
\end{abstract}
\maketitle


\section{Introduction}

The aim of this article is to study Einstein metrics (i.e., such that the  Ricci curvature is proportional to the metric) 
of Riemannian and Lorentzian  signature  in the realm of  projective geometry. 

Recall that two pseudo-Riemannian metrics $g$ and $\bar g$ on a manifold $M$ are called
\emph{projectively equivalent}\footnote{The notions ``geodesically equivalent'' or ``projectively related'' are also common.}
if their unparametrized geodesics coincide. Clearly, any constant multiple of $g$ 
 is projectively equivalent to $g$. A generic metric does not admit other examples of projectively equivalent 
metrics, see \cite{MatGenRel}. 
If two metrics $g,\bar g$ are \emph{affinely equivalent}, that is, if their Levi-Civita connections coincide, 
then they are also projectively equivalent. 
Affinely equivalent metrics are well-understood at least in Riemannian \cite{deRham,ei} and Lorentzian signature \cite{Sol,Petrov}, 
see also Lemma \ref{lem:decomp} below.   The case of arbitrary signature is   much more complicated,  see \cite{Sol}
or the more recent article \cite{Boubel} for a local description of all such metrics.

The theory of projectively equivalent metrics has a long and rich history   -- we refer to the introductions 
of \cite{KioMatEinstein,MatHyper} or to survey \cite{mikes}  for more details, and focus on   Einstein metrics in what follows.

 Einstein metrics are very natural objects in projective geometry. 
For instance, as shown in \cite{KioMatEinstein},  the property of a metric $g$ to be Einstein is projectively  
invariant in the following  sense:   
any metric that  projectively equivalent and not affinely equivalent to an Einstein metric  is also Einstein. 
A more educated point of view on the whole subject is the following: 
a projective geometry, given by a class of projectively equivalent connections (not necessarily Levi-Civita 
connections), is an example of a parabolic geometry, a special case of a Cartan geometry, see the 
monographs \cite{CapBook,Sharpe}. As shown in \cite{EastMat}, the metrics with Levi-Civita connection 
contained in the given projective class are in one-one correspondence to solutions of a certain overdetermined system of 
partial differential equations. This system  is a so-called first Bernstein-Gelfand-Gelfand equation 
\cite{CalBGG,CapBGG} and, as shown in \cite{CapGover}, Einstein metrics correspond 
to a special class of solutions called normal.

The \emph{degree of mobility $D(g)$} of a pseudo-Riemannian metric $g$ is   the dimension of 
the space of $g$-symmetric solutions of the PDE \eqref{eq:main}. As we explain in Section \ref{sec:basic},   
nondegenerate solutions of \eqref{eq:main} are in one-to-one correspondence with the metrics projectively 
equivalent to $g$. Hence, intuitively, $D(g)$ is the dimension of the space of metrics projectively equivalent to $g$. 

We have $D(g)=1$ for a generic metric $g$ and $D(g)\geq 2$ if $g$ admits a projectively equivalent metric that is 
nonproportional to $g$. As our main result, we determine all possible values for the degree of mobility $D(g)$ of 
Riemannian and Lorentzian Einstein metrics, locally or on simply connected\footnote{By definition, simply connectedness 
implies connectedness.}  manifolds. Let us denote by ``$[\alpha]$'' the
integer part of a real number $\alpha$.

\begin{thm}\label{thm:main}
Let $(M,g)$ be a simply connected Riemannian or Lorentzian Einstein manifold of dimension $n\geq 3$. 
Suppose $g$ admits a projectively equivalent
but not affinely equivalent metric. 

Then, the degree of mobility $D(g)$ is one of the numbers $\geq 2$ from the following list:
\begin{itemize}
\item $\frac{k(k+1)}{2}+l$, where $n\geq 5$, $0\leq k\leq n-4$ and $1\leq l\leq [\frac{n+1-k}{5}]$
for $g$ Riemannian and Lorentzian.\vspace{1mm}
\item $\frac{k(k+1)}{2}+l$, where $n\geq 5$, $k=n-3\mbox{ mod }5$, $2\leq k\leq n-3$ and $l=[\frac{n+2-k}{5}]$
for $g$ Lorentzian.\vspace{1mm}
\item $\frac{(n+1)(n+2)}{2}$.
\end{itemize}
Conversely, for $n\geq 3$ and each number $D\geq 2$ from this list, there exist  simply connected 
$n$-dimensional Riemannian resp. Lorentzian Einstein manifolds admitting   projectively equivalent  but 
not affinely equivalent metrics and such that  $D$ is the degree of mobility $D(g)$.
\end{thm}

\begin{figure}
  \includegraphics[width=.7\textwidth]{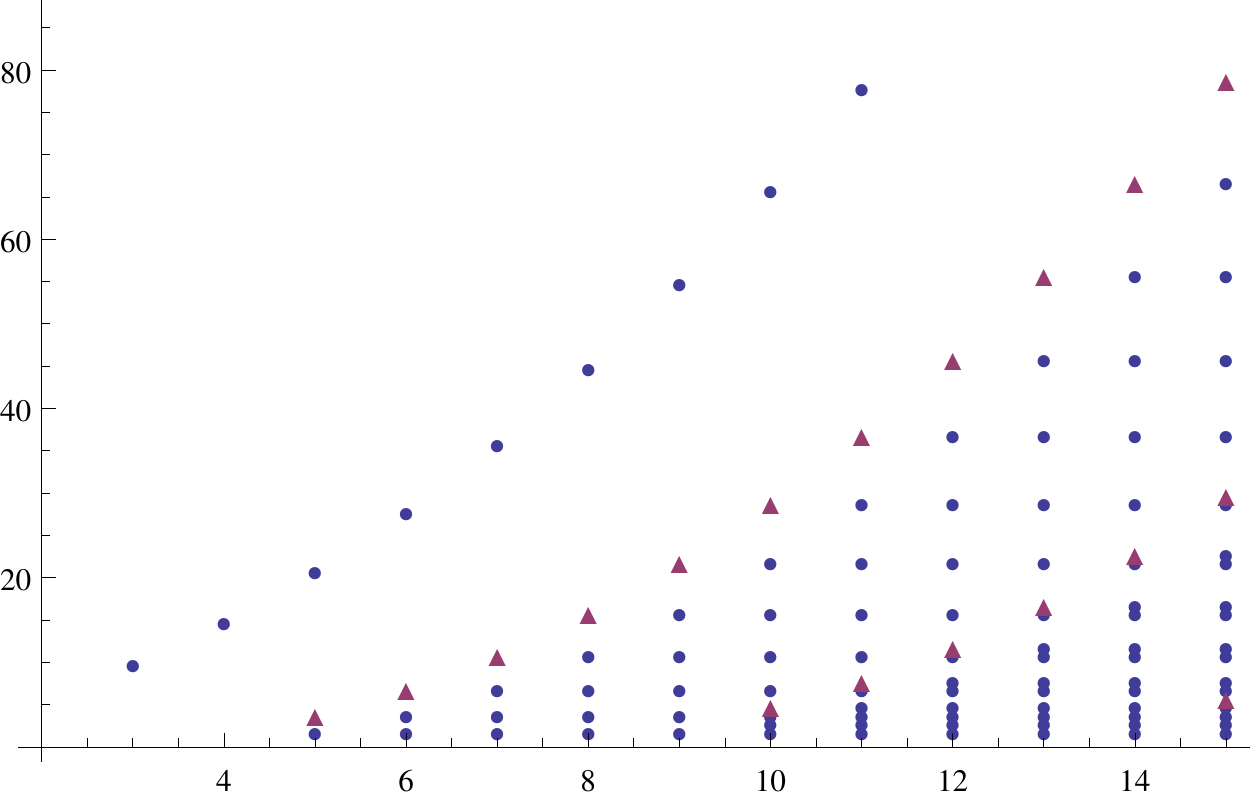}
  \caption{ Degree of mobility $D(g)$ from Theorem \ref{thm:main} for $3\leq \mathrm{dim}\,M\le 15$. 
	The triangles denote the additional values for Lorentz signature.}\label{1}
\end{figure}  

In Theorem \ref{thm:main}, the degree of mobility is at least $2$ since we assumed that
$g$ admits a metric $\bar g$ projectively equivalent to $g$ but not affinely equivalent to it. 
Suppose this assumption is dropped, that is, let us assume all metrics projectively equivalent 
to $g$ are affinely equivalent to it. In this case the complete list of   possible values of the 
degree of mobility of $g$ can be easily obtained by combining Lemma \ref{lem:decomp} below with  
methods similar to the ones used in Section \ref{sec:Bnonzero} and Section \ref{sec:Bzeromuzero}. 
It is  
$$
\{k(k+1)/2+l:0\leq k\leq n-2,1\leq l\leq [(n-k)/2]\}\cup\{n(n+1)/2\}
$$
if $g$ is Einstein with nonzero scalar curvature and  
$$
\{k(k+1)/2+l:0\leq k\leq n-4,1\leq l\leq [(n-k)/4]\}\cup\{n(n+1)/2\}
$$
 if $g$ is Ricci flat.

It is well-known, see e.g. \cite[p.134]{Sinjukov},  that if $D(g) $  is equal to its  
maximal value $(n+1)(n+2)/2$, then $g$ has constant sectional curvature. Conversely, 
this value is attained on simply connected manifolds of constant sectional curvature.
In view of this, the case  $n= 3$   in Theorem \ref{thm:main} is trivial,  since 
a $3$-dimensional  Einstein metric  has constant 
sectional curvature  and its  degree of mobility takes the maximum value $D(g)=10$.

For $4$-dimensional Einstein metrics, we obtain the following statement
as an immediate consequence of Theorem \ref{thm:main} (compare also Figure \ref{1}):
\begin{cor}\label{cor:main}
Let $(M,g)$ be a $4$-dimensional Riemannian or Lorentzian Einstein manifold. 
Suppose $\bar g$ is projectively equivalent to $g$ but not affinely equivalent. 
Then, $g$  has constant sectional curvature. 
\end{cor}

Corollary \ref{cor:main} was known before, see  \cite[Theorem 2]{KioMatEinstein} (or, alternatively, \cite{Hall2}), 
and it is actually true for metrics of arbitrary signature. However, our methods for proving Theorem \ref{thm:main}  and 
Corollary \ref{cor:main} are different from that  used in \cite{Hall2,KioMatEinstein} 
(although we will rely on some statements from \cite{KioMatEinstein}). A special case of Corollary \ref{cor:main}
was also considered in \cite{Petrov}   where it was proven that $4$-dimensional Ricci flat   nonflat   metrics cannot be 
projectively equivalent unless they are affinely equivalent. This result was generalized to Einstein metrics of arbitrary 
scalar curvature  in \cite{Hall1}. 
 Note that  by \cite[Theorem 1]{KioMatEinstein}, 
the statement of Corollary \ref{cor:main} survives for arbitrary dimension under the assumption  that both 
 metrics are geodesically complete. 

Projective equivalence of Lorentzian Einstein metrics, in particular, the problem we have 
investigated, was actively studied  in general relativity, see the classical 
references \cite{eiGR,eiBook,Weyl1} and the more recent articles \cite{Hall1,Hall2,MatGenRel}.
The motivation to study this problem is based on the description of trajectories of freely falling particles in vacuum 
as unparametrized geodesics of a Lorentzian Einstein  metric. The initial question, studied in \cite{Ehlers,Petrov,Weyl1},
is whether and under what conditions one can reconstruct the spacetime metric by only observing freely falling particles.
We study the   `freedom'  of such a reconstruction: the number of parameters  is  given by Theorem \ref{thm:main}. 

We see from Theorem \ref{thm:main} that the list for the values of the degree of mobility for Riemannian Einstein metrics 
is strictly smaller than the list for Lorentzian Einstein metrics. This difference starts in dimension five:
for a $5$-dimensional Riemannian Einstein metric $g$ we have $D(g)=1,2$ or $g$ has constant sectional curvature (i.e., $D(g)=21$). 
However, according to Theorem \ref{thm:main}, there exist $5$-dimensional Lorentzian Einstein metrics having $D(g)=4$. 
For instance, consider
\begin{ex}\label{ex:counterex1}
The nonconstant curvature metric
$$
g=\d t^2+e^{2t}(\d x_0\odot \d x_1+e^{x_2}\mathrm{sin}(x_3)\d x_1^2+\d x_2^2+\d x_3^2)
$$
on $M=\R^5$ (with coordinates $(t,x_0,x_1,x_2,x_3)$) is Einstein with scalar curvature $20$ and has signature $(1,4)$. 
In addition to $g$, the following symmetric $(0,2)$-tensors are solutions of equation \eqref{eq:main}:
$$
L_1=e^{2t}\d t^2,\,\,\,L_2=e^{2t}(x_1\d t+\d x_1)^2,\,\,\,L_3=e^{2t}\d t\odot (x_1\d t+\d x_1).
$$
\end{ex}

Without the assumption that the metric is Einstein, an analogue of Theorem \ref{thm:main} is  \cite[Theorem 1]{FedMat}.
Obviously, the values obtained in Theorem \ref{thm:main} are contained in the list of \cite[Theorem 1]{FedMat}, but our  list
is of course  thinner: not every value from \cite[Theorem 1]{FedMat} can be realized as the degree of mobility of an 
Einstein metric. We suggest to compare Figure \ref{1} above with \cite[Fig. 1]{FedMat}. 

Note also that most experts (including us) expected that the list for the values of the degree of mobility
should not depend on the signature. This is true (at least when comparing Riemannian and Lorentzian signature)
if we do consider general metrics (not necessarily Einstein), see \cite[Theorem 1]{FedMat}.
As stated in Theorem \ref{thm:main}, it is not true when we consider Einstein metrics, see also Example \ref{ex:counterex1} above.

Note that if the manifold is closed, the list of possible values for the degree of mobility is  much shorter. Indeed, by 
\cite{KioMatEinstein,MatMoun}, a metric that is 
 projectively equivalent to an Einstein metric of nonconstant sectional curvature on a closed manifold
  is affinely equivalent to it.

\subsection{Application: the dimension of the space of essential projective vector fields}

Let $(M,g)$ be a pseudo-Riemannian manifold. A diffeomorphism  $f:M\rightarrow M$ is called 
a \emph{projective transformation} if it maps unparametrized geodesics to unparametrized geodesics or, 
equivalently, if $f^{*}g$ is projectively equivalent to $g$. The isometries of $g$ are clearly  projective 
transformations. A projective transformation is called \emph{essential} if it is not an isometry of the metric. 

A vector field $v$ on   $(M,g)$ is called \emph{projective} if its local flow
consists of projective transformations. A projective vector field is called \emph{essential} if it is not 
a Killing vector field. 

Let $\mathfrak{p}(g)$ and $\mathfrak{i}(g)$ denote the vector spaces (in fact, Lie algebras) of projective
and Killing vector fields respectively. The quotient $\mathfrak{p}(g)/\mathfrak{i}(g)$ will be referred to as the 
\emph{space of essential projective vector fields}. In the generic case, see Remark \ref{rem:essprojvf}
below, this space can be naturally identified with a subspace (thought, not a subalgebra) of $\mathfrak{p}(g)$.

We determine all possible values for the dimension of the space of essential projective vector fields of 
a Riemannian or Lorentzian Einstein metric:
\begin{thm}\label{thm:proj}
Let $(M,g)$ be a simply connected Riemannian or Lorentzian Einstein manifold of dimension $n\geq 3$
 which admits a metric that is projectively equivalent but not affinely equivalent to $g$. 
Then, the possible values for the dimension of the space of essential projective vector fields are given by the 
numbers $\geq 1$ from the following list:
\begin{itemize}
\item $\frac{k(k+1)}{2}+l-1$, where $n\geq 5$, $0\leq k\leq n-4$ and $1\leq l\leq [\frac{n+1-k}{5}]$
for $g$ Riemannian and Lorentzian.\vspace{1mm}
\item $\frac{k(k+1)}{2}+l-1$, where $n\geq 5$, $k=n-3\mbox{ mod }5$, $2\leq k\leq n-3$ and $l=[\frac{n+2-k}{5}]$
for $g$ Lorentzian.\vspace{1mm}
\item $\frac{(n+1)(n+2)}{2}-1$.
\end{itemize}
Conversely, for $n\geq 3$ and each number $\geq 1$ from this list, there exists a $n$-dimensional simply connected 
Riemannian resp. Lorentzian Einstein metric admitting a projectively equivalent but not affinely equivalent metric and
for which this number is the dimension of the space of essential projective vector fields.
\end{thm}
Comparing the list  from  Theorem \ref{thm:proj} with that in Theorem \ref{thm:main}, 
we see that the possible values for $\mathrm{dim}\left(\mathfrak{p}(g)/\mathfrak{i}(g)\right)$
are given by the values for the degree of mobility $D(g)$ subtracted by $1$.  Indeed, in the generic case,
 the number of essential projective vector fields of an Einstein metric is  $D(g)-1$. 
   Moreover, if in addition to our assumptions the metric is Riemannian or 
   the scalar curvature is not zero, then  there exists 
 a natural linear mapping with $1$-dimensional kernel  from the set of solutions of \eqref{eq:main} to  the space 
$ \mathfrak{p}(g)/\mathfrak{i}(g)$,   
 see  Section \ref{sec21} below. There exist though Einstein 
 metrics of Lorentzian signature such that  $\mathrm{dim}\left(\mathfrak{p}(g)/\mathfrak{i}(g)\right)< D(g)-1$.

By Theorem \ref{thm:proj}, any Einstein metric of Riemannian or Lorentzian signature 
admitting a nonaffinely equivalent projectively equivalent metric also admits an essential projective vector field. 
The next theorem shows that the assumption on signature is not essential. 

\begin{thm}\label{thm:proj2}
Let $g$ be an Einstein metric of arbitrary signature on a simply connected 
 manifold   of dimension $n\geq 3$. 
If there exists a metric that is projectively equivalent  but  not affinely equivalent to $g$,  
  there exists at least one essential projective vector field for $g$. 
\end{thm}

Examples show that the assumption that the metric is Einstein is essential for Theorem \ref{thm:proj2}. 

As we already recalled  above,  an  Einstein metric of arbitrary signature and of nonconstant sectional curvature 
on a closed manifold does not admit projectively  but not affinely equivalent metrics. Therefore, on a closed Einstein manifold 
of nonconstant sectional curvature every projective transformation is an affine transformation and, hence, every projective 
vector field is an affine vector field. 
Actually, in the Riemannian case we do not need the assumption that the metric is Einstein in the latter statement, 
see \cite[Corollary 1]{MatLichOb}.

Similar results were also obtained in the case the manifold is not necessarily closed but under the additional 
assumption that  the metric $g$ and a projectively equivalent but not affinely equivalent metric $\bar g$ are  complete.  
By \cite[Theorem 1]{KioMatEinstein},   projective but not affine  equivalence of two complete metrics 
(of arbitrary signature)  one of which is Einstein implies that both   metrics have  constant sectional curvature.  
This implies that complete   Einstein metrics do not admit complete projective but not affine vector fields.  
Again in the Riemannian case we do not need the assumption that the metric is Einstein in the latter statement, 
see \cite[Theorem 1]{MatLichOb}.

Note that  the result of Theorem \ref{thm:proj} has a predecessor: in \cite[Theorem 3]{FedMat} the possible
dimensions of the space of essential projective vector fields have been determined for a general
Riemannian or Lorentzian metric. As before the list of values we have obtained in the Einstein case
is shorter than the list of values obtained in \cite[Theorem 3]{FedMat}.

\subsection{Organisation of the article}

In Section \ref{sec:basic}, we recall basic facts from the theory of projectively equivalent metrics.

The remaining sections deal with the proofs of the Theorems \ref{thm:main}, \ref{thm:proj} and \ref{thm:proj2}. 
As mentioned
above,  the case of general  (= not necessarily Einstein) metrics was solved in \cite{FedMat}. 
We extensively use and therefore quote necessary results  from \cite{FedMat} in the paper  and  indicate the  
places when the additional condition that the metric is Einstein becomes important.  

The proof of Theorem \ref{thm:main} will be given in Section \ref{sec:proofmain}. It is divided into several
parts and a rough discription of how we proceed can be found in Section \ref{sec:scheme}.

The proof of Theorem \ref{thm:proj} and that of Theorem \ref{thm:proj2} will be given in Section \ref{sec:proofproj}. 

\section{Basic formulas}
\label{sec:basic}

Let $g,\bar g$ be two pseudo-Riemannian metrics on an $n$-dimensional manifold $M$.
We define a symmetric nondegenerate $(0,2)$-tensor $L$ by 
\begin{align}
L=L(g,\bar g)=\Big|\frac{\mathrm{det}\,\bar g}{\mathrm{det}\, g}\Big|^{\frac{1}{n+1}}g\bar g^{-1}g.\label{eq:defL}
\end{align}
In the formula above, we view $g,\bar g:TM\rightarrow T^* M$ naturally as bundle isomorphisms and identify
$(0,2)$-tensors with endomorphism $TM\rightarrow T^* M$ via $L(X)(Y)=L(X,Y)$ for $X,Y\in TM$. In tensor notation, \eqref{eq:defL}
reads
$$
L_{ij}=\Big|\frac{\mathrm{det}\,\bar g}{\mathrm{det}\, g}\Big|^{\frac{1}{n+1}}g_{ik}\bar g^{kl}g_{lj},
$$
where $\bar g^{ik}\bar g_{kj}=\delta^i_j$. It is a fundamental fact, see \cite{Sinj}, that $g$ and 
$\bar g$ are projectively equivalent, if and only if the tensor 
$L$ from \eqref{eq:defL} is a solution to the following PDE 
\begin{align}
\nabla_X L=X^\flat\odot \Lambda,\,\,\,X\in TM,\label{eq:main}
\end{align}
where $\Lambda$ is a certain $1$-form, $\nabla$ denotes the Levi-Civita connection of $g$, 
$\alpha\odot\beta=\alpha\otimes\beta+\beta\otimes\alpha$ for $1$-forms $\alpha,\beta$ and $X^\flat=g(X,.)$ 
denotes the metric dual w.r.t. $g$. 

Throughout the article, when it is clear which metric is used, we will denote by $X^\flat\in T^*M$ the metric dual 
of a vector $X\in TM$ and by $\alpha^\sharp\in TM$ the metric dual of a $1$-form $\alpha\in T^*M$. Similarly, for 
a $(0,2)$-tensor $L$ we let $L^\sharp$ denote the corresponding $(1,1)$-tensor defined by $g(L^\sharp.,.)=L$.  

Taking a trace in \eqref{eq:main} using $g$ shows that 
$$
\Lambda=\d\lambda\mbox{, where }\lambda=\frac{1}{2}\mathrm{trace}(L^\sharp).
$$
Thus, \eqref{eq:main} is in fact a linear PDE of first order on symmetric $(0,2)$-tensors $L$. As stated above, 
the nondegenerate symmetric solutions of \eqref{eq:main} correspond via \eqref{eq:defL} to metrics projectively 
equivalent to $g$. In fact, if $L$ is such a solution then $\bar g=(\mathrm{det}\,L^\sharp)^{-1}g((L^\sharp)^{-1}.,.)$ is 
projectively equivalent to $g$. Since $g$ is always a solution of \eqref{eq:defL} 
(corresponding to the fact that $g$ is projectively equivalent to itself), we can (locally) make any symmetric 
solution of \eqref{eq:main} nondegenerate by adding a suitable multiple of $g$. In this sense the linear space 
of symmetric solutions of \eqref{eq:main} corresponds to the space of metrics being projectively equivalent to $g$.
\begin{defn}
Let $(M,g)$ be a pseudo-Riemannian manifold. We denote by  $\A(g)$  the linear space of 
symmetric solutions of \eqref{eq:main}. The \emph{degree of mobility $D(g)$} of $g$ is the dimension of $\A(g)$.
\end{defn}
In view of the above correspondence we will often consider a pair $g,L$, where $L\in \A(g)$,
instead of a pair $g,\bar g$ of projectively equivalent metrics. 

As stated in the introduction, affinely equivalent metrics (i.e. metrics having the same Levi-Civita connections)
are projectively equivalent. Obviously, two metrics $g,\bar g$ are affinely equivalent if and only if 
the tensor $L=L(g,\bar g)$ from \eqref{eq:defL} is parallel (w.r.t. the Levi-Civita connection of one of the metrics). 
In view of  \eqref{eq:main}, this is equivalent to 
the property that 
$\Lambda$ from  \eqref{eq:main} is identically zero. Combining these, we obtain the following wellknown statement:  
\begin{lem}\label{lem:affine}
Let $g,\bar g$ be projectively equivalent pseudo-Riemannian metrics on a manifold $M$ and let $L=L(g,\bar g)\in \A(g)$
be given by \eqref{eq:defL}. Then, $g,\bar g$ are affinely equivalent if and only if $L$ is $g$-parallel 
if and only if the $1$-form $\Lambda$ corresponding to $L$ is identically zero.
\end{lem}

Of fundamental importance for our goals is the following
\begin{thm}\cite{KioMatEinstein}\label{thm:extsys}
Let $(M,g)$ be a connected pseudo-Riemannian Einstein manifold of dimension $n\geq 3$ such that 
at least one $L\in \A(g)$ is nonparallel. Let 
$$
B=-\frac{\mathrm{Scal}}{n(n-1)},
$$
where $\mathrm{Scal}$ denotes the scalar curvature of $g$. 

Then, for every $L\in \A(g)$ with corresponding $1$-form $\Lambda$, 
there exists a function $\mu$ such that $(L,\Lambda,\mu)$ satisfies
\begin{align}
\nabla_X L=X^\flat\odot \Lambda,\,\,\,\nabla\Lambda=\mu g +BL,\,\,\,\nabla\mu=2B\Lambda.
\label{eq:extsys}
\end{align}
\end{thm}
\begin{rem}
Theorem \ref{thm:extsys} follows from \cite[Corollary 1 and 2]{KioMatEinstein}. 
As shown in \cite{KioMat}, under the assumption $D(g)\geq 3$, the statement is actually 
true for any metric (not necessarily Einstein) and a certain constant $B$ (which is 
not necessarily equal to $-\mathrm{Scal}/n(n-1)$ in this case). 
\end{rem}

\section{Proof of Theorem \ref{thm:main}}
\label{sec:proofmain}

\subsection{Scheme of the proof}
\label{sec:scheme}

By Theorem \ref{thm:extsys}, under the assumptions of Theorem \ref{thm:main}, the degree of mobility $D(g)$ equals 
the dimension of the space of solutions of the system \eqref{eq:extsys}. 
The proof of Theorem \ref{thm:main}  is different for $B=-\mathrm{Scal}/n(n-1)=0$ and for $B\ne 0$.  

Consider first the case  $B\neq 0$. By scaling the metric $g$ we may assume that $B=-1$.   
The key observation  is that for $B=-1$ the solutions 
of  the system \eqref{eq:extsys}  correspond to parallel 
symmetric $(0,2)$-tensors on the metric cone $(\hat M:= \mathbb{R}_{>0}\times M,\hat g:= \d r^2+ r^2 g )$ 
over $(M,g)$. Depending on the sign of the initial $B$ and on the signature of the metric $g$, the metric 
cone $(\hat M,\hat g)$ has signature $(0,n+1)$, $(1,n)$, $(n,1)$,  
or $(n-1,2)$.   
 The space of parallel  tensors for cone metrics of these signatures  
has been described in \cite{FedMat}. The assumption that the initial metric is  Einstein is equivalent to 
the condition that the cone metric is Ricci-flat. Combining  the description of parallel tensors with the
Ricci-flat condition, we obtain   
the list of possible values for $D(g)$.

Consider now the case when 
 $B =0$ but assume that at least one solution of \eqref{eq:extsys}
has $\mu\neq 0$. This case is treated in Section \ref{sec:Bzero}. We show  
the local existence of  an Einstein metric
$\bar g$ of the same signature as $g$ and projectively equivalent to $g$ such that the corresponding constant  
$\bar B$ for $\bar g$ is nonzero.  This allows to reduce the problem to the already solved one. 

The remaining case, considered in Section \ref{sec:Bzeromuzero}, is when $B =0$ and $\mu = 0$ for all solutions 
of \eqref{eq:extsys}. In this case additional work is necessary, but also here the problem reduces  
to determining the dimension of the space of parallel symmetric $(0,2)$-tensors (although, this time, 
we consider such tensors for $g$ and not
for the cone metric $\hat g$). We can locally describe all such metrics
and the Einstein condition poses additional restrictions on the possible values
of the degree of mobility.

Finally, in Section \ref{sec:realization} we complete the proof of Theorem \ref{thm:main} by showing that actually
each number $D$ from the list in the theorem can be realized as the degree of mobility of a certain 
Lorentzian resp. Riemannian  Einstein 
metric. This is done by going in the opposite direction of the procedure explained in Section \ref{sec:Bnonzero}: 
we construct a Ricci flat cone such that the space of parallel symmetric $(0,2)$-tensor fields has dimension equal to $D$.

\subsection{The case of nonzero scalar curvature}
\label{sec:Bnonzero}

The goal of this section is to prove
\begin{prop}\label{prop:Bnonzero}
Let $(M,g)$ be a simply connected Riemannian or Lorentzian Einstein manifold of dimension $n\geq 3$
with nonzero scalar curvature such that $\Lambda\neq 0$ for at least one solution of the system \eqref{eq:extsys}. 

Then, the degree of mobility $D(g)$ is given by one of the values in the list of Theorem \ref{thm:main}.
\end{prop}
We will go along the same line of ideas as in \cite[Section 4]{FedMat}. We will start working with a general
Riemannian or Lorentzian metric $g$ and implement the condition that $g$ is Einstein at the corresponding places.
Since the constant $B:=-\mathrm{Scal}/n(n-1)$ in \eqref{eq:extsys} is nonzero, we can consider the metric $-Bg$ 
instead of $g$ and for simplicity, 
we denote this new metric by the same symbol $g$. Because we have rescaled the metric, the system \eqref{eq:extsys}
is now satisfied for a new constant $B=-1$, that is, for every $L\in \A(g)$ with corresponding $1$-form $\Lambda$,
we find a function $\mu$ such that $(A,\Lambda,\mu)$ satisfies
\begin{align}
\nabla_X L=X^\flat\odot \Lambda,\,\,\,\nabla\Lambda=\mu g -L,\,\,\,\nabla\mu=-2\Lambda.
\label{eq:extsys-1}
\end{align}
Note that since the new metric $g$ and the original metric are proportional to each other, they have the same degree of mobility.

Note also that since the initial metric was assumed to be Riemannian or Lorentzian the signature of the new metric $g$ 
is now $(0,n)$, $(1,n-1)$, $(n,0)$ or $(n-1,1)$, depending on the sign of the scaling constant $B$. 

For further use let us recall the following  statement which can be found for example in 
\cite[Proposition 3.1]{MatMoun} or \cite[Theorem 8]{FedMat} and can be verified by a 
direct calculation.
\begin{lem}\label{lem:isom}
There is an isomorphism between the space of solutions of \eqref{eq:extsys-1} on a 
pseudo-Riemannian manifold $(M,g)$ and the space of parallel symmetric $(0,2)$-tensors on the 
metric cone $(\hat M=\R_{>0}\times M,\,\,\,\hat g=\d r^2+r^2 g)$ over $(M,g)$. 
\end{lem}
Since the manifold $(M,g)$ in our case has signature $(0,n)$, $(1,n-1)$, $(n,0)$ or $(n-1,1)$, the signature 
of the metric $\hat g$ is $(0,n+1)$, $(1,n)$, $(n,1)$ or $(n-1,2)$. 

By Lemma \ref{lem:isom}, in order to  determine the possible values of the degree of mobility $D(g)$ of $g$,  
it is sufficient to  calculate the
possible dimensions of the space of parallel symmetric $(0,2)$-tensors for the cone metric $\hat g$. 

The description of such tensors has been obtained in \cite[Theorem 5]{FedMat}. Since we will come back to this result
later on, we summarize it in
\begin{lem}\label{lem:decomp}
Let $(M,g)$ be a simply connected $n$-dimensional pseudo-Riemannian manifold. 
Assume one of the following:
\begin{enumerate}
\item $g$ has signature $(0,n)$ or $(1,n-1)$.
\item $g$ is a metric cone of signature $(n-2,2)$.
\end{enumerate} 
Consider the maximal holonomy decomposition
\begin{align}
T M=V_0\oplus V_1\oplus...\oplus V_l\label{eq:decomp}
\end{align}
of the tangent bundle $TM$ into mutually orthogonal subbundles invariant w.r.t. the holonomy group 
$H( g)$ of $ g$. More precisely, $V_0$ is flat in the sense that $H( g)$ acts trivially on it 
and $V_1,...,V_l$ are indecomposable, i.e., do not admit an invariant nondegenerate subbundle. 
Let $g_i$ denote the restriction of $g$ to $V_i$ for $i=0,...,l$. If $\tau_1,...,\tau_k$ is a basis 
for the space of parallel $1$-forms for $ g$, then any parallel symmetric $(0,2)$-tensor 
can be written as
\begin{align}
\sum_{i,j=1}^k c_{ij}\tau_i\otimes \tau_j+\sum_{i=1}^l c_i g_i\label{eq:parallel}
\end{align}
for constants $c_{ij}=c_{ji}$ and $c_i$. 
\end{lem}
\begin{rem}
The statement of Lemma \ref{lem:decomp} is classical for positive definite $g$ \cite{ei} and 
for Lorentzian signature \cite{eiparallel,Sol}. The description \eqref{eq:parallel} of parallel symmetric 
$(0,2)$-tensors for metric cones of signature $(n-2,2)$ is given by \cite[Theorem 5]{FedMat}. 
If the metric is not a cone the description of  such tensors for metrics of arbitrary signature
is in general much more complicated, see \cite{Boubel}. 
\end{rem}
Formula \eqref{eq:parallel} shows that the dimension of parallel symmetric $(0,2)$-tensors for $\hat g$ and, 
hence, the degree of mobility $D(g)$ of $g$, is given by 
\begin{align}
D(g)=\frac{k(k+1)}{2}+l,\label{eq:degree}
\end{align}
where $k$ is the number of linearly independent parallel vector fields for $\hat g$ and $l$ the number of 
indecomposable components in the holonomy decomposition of $(\hat M,\hat g)$. To prove the first direction of 
Theorem \ref{thm:main} under the assumption $B\neq 0$, it therefore suffices to
determine the range of the integers $k,l$ in \eqref{eq:degree}. We start listing some known facts 
concering curvature properties of the metric cone.
\begin{lem}\label{lem:curvfacts}
Let $(\hat M,\hat g)$ be the metric cone over an $n$-dimensional pseudo-Riemannian manifold $(M,g)$.
Then, the following statements hold:
\begin{enumerate}
\item $\hat g$ is flat if and only if $g$ has constant sectional curvature equal to $1$.
\item $\hat g$ is Ricci flat if and only if $g$ is Einstein with scalar curvature $n(n-1)$.
\end{enumerate}
\end{lem}
\begin{proof}
The statements follow from the usual formulas relating the curvatures of $\hat g$ and $g$, 
see for instance \cite[equation (3.2)]{ACGL}. 
\end{proof}
Since in our case the given Einstein metric $g$ has $B=-1$, we have $\mathrm{Scal}(g)=n(n-1)$ and therefore 
$\hat g$ is Ricci flat. 

The so-called \emph{cone vector field} $\xi=r\partial_r$ on $\hat M$ satisfies 
\begin{align}
\hat \nabla \xi=\Id.\label{eq:conevf}
\end{align}
This is straight-forward to see (using the formulas for the Levi-Civita connection $\hat\nabla$ 
of $\hat g$, see for instance \cite[equation (3.1)]{ACGL}) and is wellknown, see \cite[Lemma 1]{FedMat}. 
A manifold $(\hat M,\hat g)$ admitting a vector field $\xi$ satisfying \eqref{eq:conevf}
will be called a \emph{local cone} in what follows. The name is justified in
\begin{lem}\label{lem:localcone}
Let $(\hat M,\hat g,\xi)$ be a local cone of dimension $n+1$. Then, $\xi$ is nonvanishing on a dense 
and open subset and in a neighbourhood of each point of this subset $(\hat M,\hat g,\xi)$ takes the form
$$
\hat M=\R_{>0}\times M,\,\,\,\hat g=\varepsilon \d r^2+r^2 g,\,\,\,\xi=r\partial_r
$$
where $(M,g)$ is a certain $n$-dimensional pseudo-Riemannian manifold and $\varepsilon=\mathrm{sgn}(\hat g(\xi,\xi))$. 
That is, locally in a neighbourhood of almost every point, $(\hat M,\hat g)$ is a metric cone,
up to multiplication by $-1$, over a certain pseudo-Riemannian manifold.
\end{lem}
\begin{proof}
The statement and its proof are standard, see \cite[Lemma 1 and Remark 2]{FedMat} (the role of the positive 
function $v$ used in this reference is played by $\frac{1}{2}\hat g(\xi,\xi)$ for $\hat g(\xi,\xi)>0$).
\end{proof}

We will need a dimensional estimate for nonflat Ricci flat local cones.
\begin{lem}\label{lem:dimRicciflat}
Let $(\hat M,\hat g,\xi)$ be a Ricci flat local cone. 
\begin{enumerate}
\item If $\hat g$ is nonflat, then $\mathrm{dim}\,\hat M\geq 5$.
\item If $\hat g$ is nonflat and $u$ is a nonzero parallel null vector field for $g$, 
then $\mathrm{dim}\,\hat M\geq 6$.
\end{enumerate}
\end{lem}
\begin{proof}
$(1)$ follows immediately from Lemma \ref{lem:curvfacts}: locally, in a neighborhood
of almost every point, $(\hat M, \hat g)$ is a cone over an Einstein manifold $(M,g)$ 
of dimension $n$ (where $\mathrm{dim}\,\hat M=n+1$) with scalar curvature $\mathrm{Scal}(g)=n(n-1)$. 
If $n+1= 4$, $g$ is a $3$-dimensional Einstein metric and therefore has constant 
sectional curvature equal to $1$. This, in turn, implies $\hat  g$ is flat. 

$(2)$ Let $u$ be a nonzero parallel null vector field for $\hat g$. Suppose $\hat g(u,\xi)=0$
on some open subset $U$. Taking the derivative of this equation and using \eqref{eq:conevf},
we obtain $\hat g (u,.)=0$ on $U$, hence, $u=0$ on $U$, a contradiction. On the other hand, 
suppose $\xi=fu$ on some open subset $U$ for a smooth function $f:U\rightarrow \R$.
Again, taking the covariant derivative of this equation and using \eqref{eq:conevf}, we obtain
$\Id=\d f\otimes u$ which is clearly a contradiction (since the endomorphism on the right-hand
side has rank $1$). We obtain that at every point $p$ of an open and dense subset of $\hat M$, 
$\xi$ and $u$ are linearly independent (see also \cite[Lemma 3]{FedMat}) and $\hat g(u,\xi)(p)\neq 0$.
Then, $\hat g$ is nondegenerate on $\mathrm{span}\{\xi(p),u(p)\}$. If $\hat M\leq 5$, the statement that
$\hat R(p)=0$ now reduces to the statement that Ricci flat curvature operators in dimensions $\leq 3$
are flat.
\end{proof}

The following example shows that the existence of two linearly independent
parallel vector fields on a Ricci flat cone $(\hat M,\hat g)$ of dimension $6$
does in general not imply that $\hat g$ is flat:
\begin{ex}\label{ex:counterex}
The cone metric over the metric from Example \ref{ex:counterex1}, given by 
\begin{align}
\hat g=\d r^2+r^2[-\d t^2+e^{2t}(\d x_0\odot \d x_1+e^{x_2}\mathrm{sin}(x_3)\d x_1^2-\d x_2^2-\d x_3^2)],
\label{eq:metriccounterex}
\end{align}
has signature $(4,2)$ and is indecomposable nonflat and Ricci flat. 
It admits two linearly independent parallel vector fields
\begin{align}
v_1=e^t(\partial_r-\frac{1}{r}\partial_t),\,\,\,
v_2=x_1e^t\partial_r+\frac{1}{r}\left(-x_1 e^t\partial_t+e^{-t}\partial_{x_0}\right)\label{eq:vfcounterex}
\end{align}
such that $\mathrm{span}\{v_1,v_2\}$ is totally isotropic.
\end{ex}
\begin{rem}
Example \ref{ex:counterex} is a special case of the following general description (which can be obtained
in a straight-forward way by applying, for instance, results of \cite{BolMat}):
any cone $(\hat M=\R_{>0}\times M,\hat g=\d r^2+r^2 g)$ with nonzero parallel null vector field $v$,
is locally of the form 
$$
\hat M=\R_{>0}\times \R\times N,\,\,\,\hat g=\d r^2+r^2(-\d t^2+e^{2t}h),\,\,\,v=e^t(\partial_r-\frac{1}{r}\partial_t),
$$
where $(N,h)$ is a certain pseudo-Riemannian manifold. We have that $\hat g$ is Ricci flat (resp. flat)
if and only if $h$ is Ricci flat (resp. flat). If $V$ is another parallel vector field for $\hat g$,
we obtain
$$
V=\left(F e^{t}-\frac{C}{2}e^{-t}\right)\partial_r+\frac{1}{r}\left(-\left(F e^t+\frac{C}{2}e^{-t}\right) 
\partial_t+e^{-t} \mathrm{grad}_h F\right)
$$
for a certain constant $C$ and a function $F$ on $N$ satisfying
$$
\nabla^h\nabla^h F=C h,
$$
where $\nabla^h$ denotes the Levi-Civita connection of $h$.
Since $\hat g(V,V)=-2CF+h(\mathrm{grad}_h F,\mathrm{grad}_h F)$ and $\hat g(v,V)=-C$,
we see that $V$ is null and perpendicular to $v$ if and only if $\mathrm{grad}_h F$  
is a parallel null vector field on $N$. To construct Example \ref{ex:counterex}, it remains to find an 
example of a nonflat Ricci flat Lorentz manifold admitting a nonzero parallel gradient null vector field.
Such metrics are described by Walker coordinates \cite{eiparallel,Walker}.
\end{rem}
As explained above, the maximal value $D(g)=(n+1)(n+2)/2$ for the degree of mobility is attained if and only if  
$g$ has constant sectional curvature, i.e., if and only if $\hat g$ is flat. Thus, in order to seek for the 
submaximal values of $D(g)$, we may assume that $\hat g$ is nonflat, i.e. $l\geq 1$ in the decomposition \eqref{eq:decomp}.
Thus, $(\hat M,\hat g)$ is a Ricci flat but nonflat cone with $k$ parallel vector fields. 
Let $\hat p\in \hat M$ be a point and denote by $M_i$ the integral leaf containing $\hat p$ of the distribution $V_i$. Then, 
$(\hat M,\hat g)$ is locally the direct product
$$
\hat M=M_0\times M_1\times...\times M_l,\,\,\,\hat g=g_0+g_1+...+g_l
$$
and, since $\hat g$ is Ricci flat, each of the metrics $g_1,...,g_l$ is Ricci flat as well ($g_0$ is the flat metric 
by construction). We recall
\begin{lem}\cite[Lemma 4 and Lemma 5]{FedMat}\label{lem:product}
Let $(\hat M,\hat g)=(M_1,g_1)\times (M_2,g_2)$ be a product of pseudo-Riemannian manifolds $(M_i,g_i)$, $i=1,2$.
Then, $(M,g)$ is a local cone if and only if both $(M_1,g_1)$ and $(M_2,g_2)$ are local cones. The cone vector fields
$\xi$ of $(\hat M,\hat g)$, $\xi_1$ of $(M_1,g_1)$ and $\xi_2$ of $(M_2,g_2)$ are related by $\xi=\xi_1+\xi_2$.
\end{lem}
\begin{proof}
Let $\xi=\xi_1+\xi_2$ be the orthogonal decomposition of the cone vector field $\xi$ of $(\hat M,\hat g)$ w.r.t. the decomposition
$T\hat M=TM_1\oplus TM_2$. For $X_1\in TM_1,X_2\in TM_2$, we obtain $X_i=\hat \nabla_{X_i}\xi=\hat\nabla_{X_i}\xi_1+\hat\nabla_{X_i}\xi_2$.
Since $\hat \nabla_{X_i} \xi_1\in TM_1$ and $\hat \nabla_{X_i} \xi_2\in TM_2$, we obtain $\hat \nabla_{X_1}\xi_2=\hat \nabla_{X_2}\xi_1=0$. Hence,
$\xi_1$, $\xi_2$ are vector fields on $M_1$ resp. $M_2$ and $\nabla^i\xi_i=\Id_{TM_i}$, $i=1,2$. Thus, $\xi_1,\xi_2$ are cone vector fields for
$(M_1,g_1)$ resp. $(M_2,g_2)$.

Conversely, if $\xi_i$ is a cone vector field for $(M_i,g_i)$, $i=1,2$, then, clearly, $\xi=\xi_1+\xi_2$ is a cone vector field for $(\hat M,\hat g)$.
\end{proof}
From Lemma \ref{lem:product} we conclude that each $(M_i,g_i)$, $i=1,...,l$, is a nonflat Ricci flat local cone which is indecomposable
by construction.

\medskip
Before we determine the range of the integer $l$ in the formula \eqref{eq:degree} for the degree of mobility $D(g)$, we introduce
some notation. 
For $i=1,...,l$ let $k_i$ denote the dimension of the space $\mathrm{Par}_i$ of parallel vector fields for $\hat g$ which take 
values in $V_i$. Obviously, when restricted to the integral leaf $M_i$, each vector field in $\mathrm{Par}_i$ is a parallel vector field
on $M_i$ for the metric $g_i$. Since $V_i$ is indecomposable, any linear combination 
of vector fields in $\mathrm{Par}_i$ must be a null vector, that is, at each point, the values of the vector fields in $\mathrm{Par}_i$
span a totally isotropic subspace of the tangent space. Since the only possible signatures of $\hat g$ are $(0,n+1)$, $(1,n)$, 
$(n,1)$ or $(n-1,2)$, we therefore have $0\leq k_1+...+k_l\leq 2$. Moreover, since by definition, $k$ 
is the number of parallel vector fields for $\hat g$, we have 
$k=\mathrm{dim}\,V_0+k_1+...+k_l$.

\medskip
To determine the range of $l$, we consider two different case:

\emph{Case 1: Suppose $0\leq k_1+...+k_l\leq 1$.} Note that this is the only case which occurs
when the initial metric $g$ is Riemannian (where ``initial'' means before multiplication with $B\neq 0$) 
-- in this case $\hat g$ cannot have signature $(n-1,2)$ and therefore $k_i<2$ for all $i=1,...,l$. 
Applying Lemma \ref{lem:dimRicciflat}, we obtain $\mathrm{dim}\,V_i=\mathrm{dim}\,M_i\geq k_i+5$ 
for $i=1,...,l$ and therefore
$$
n+1=\mathrm{dim}\,V_0+\mathrm{dim}\,V_1+...+\mathrm{dim}\,V_l\geq \mathrm{dim}\,V_0+k_1+...+k_l+5l=k+5l.
$$
Hence, $1\leq l\leq [\frac{n+1-k}{5}]$. Since there is at least one indecomposable component in 
the decomposition \eqref{eq:decomp} and this component is at least $5$-dimensional, we obtain 
$0\leq k\leq \mathrm{dim}\,\hat M-5=n-4$. In particular, this completes the proof of Proposition 
\ref{prop:Bnonzero} in case that $g$ is positive definite. 

\emph{Case 2: Suppose $k_1=2$ for the component $(M_1,g_1)$.} In this case $\hat g$ necessarily has signature 
$(n-1,2)$ and therefore also $g_1$ has signature $(\mathrm{dim}\,V_1-2,2)$. Consequently, the 
remaining components $g_0,g_2,...,g_l$ are negative definite. In particular, we have
$k_i=0$ for $i=2,...,l$ and Lemma \ref{lem:dimRicciflat} implies $\mathrm{dim}\,V_i\geq 5$ for $i=2,...,l$. 
From Example \ref{ex:counterex} we have learned that $V_1$ is at least $6$ dimensional. Using this, we obtain
$$
n+1=\mathrm{dim}\,V_0+\mathrm{dim}\,V_1+...+\mathrm{dim}\,V_l\geq \mathrm{dim}\,V_0+6+5(l-1)=k-1+5l.
$$
Hence, $1\leq l\leq [\frac{n+2-k}{5}]$. Since $0\leq \mathrm{dim}\,V_0\leq \mathrm{dim}\,\hat M-6=n-5$ and $k=\mathrm{dim}\,V_0+2$, we
obtain $2\leq k\leq n-3$. Comparing this with the first case above, the additional values for $D(g)$ appearing in the second case
occur for any $k$ in $2\leq k\leq n-3$ satisfying $k=n-3\mbox{ mod }5$ and for $l= [\frac{n+2-k}{5}]$. 
This completes the proof of Proposition \ref{prop:Bnonzero}.

\subsection{The case when the scalar curvature is zero and $\mu\neq 0$ for at least one solution of \eqref{eq:extsys}}
\label{sec:Bzero}

In this section, we prove the first direction of Theorem \ref{thm:main} for a simply connected 
Riemannian or Lorentzian Einstein manifold  $(M,g)$ such that at least one solution 
$(L,\Lambda,\mu)$ of \eqref{eq:extsys} with $B=0$ has $\mu\neq 0$.

We reduce the proof locally to Proposition \ref{prop:Bnonzero} by applying the following lemmas:
\begin{lem}\cite[Lemma 11]{FedMat}\label{lem:changeofmetric}
Let $(M,g)$ be a pseudo-Riemannian manifold. Assume one of the following:
\begin{enumerate}
\item $g$ is Riemannian and at least one solution $(L,\Lambda,\mu)$ of \eqref{eq:extsys} with $B=0$ has $\Lambda\neq 0$.
\item $g$ is Lorentzian and at least one solution $(L,\Lambda,\mu)$ 
of \eqref{eq:extsys} with $B=0$ has $\mu\neq 0$.
\end{enumerate}
Then, on each open subset with compact closure, there exists a metric $\bar g$ of the same signature as $g$ 
which is projectively equivalent to $g$ and such that the constant $\bar B$ for the system \eqref{eq:extsys} 
corresponding to $\bar g$ is nonzero.
\end{lem}
\begin{rem}\label{rem:changeofmetric}
Actually, \cite[Lemma 11]{FedMat} only contains the statement for Lorentzian signature. However, under 
the assumption of $(1)$, one can always construct a solution to \eqref{eq:extsys} such that $\mu\neq 0$ and
then the proof of \cite[Lemma 11]{FedMat} applies. 
Indeed, let $(L,\Lambda,0)$ be a solution of \eqref{eq:extsys} (with $B=0$) such that
$\Lambda\neq 0$. Let $\lambda$ be a function 
such that $\Lambda=\d\lambda$. It is easy to check that the $1$-form $\tilde\Lambda=L(\Lambda^\sharp.,.)-\lambda\Lambda$ 
satisfies $\nabla \tilde\Lambda=\tilde\mu g$ for the nonzero constant
$\tilde \mu=|\Lambda|^2$. Then, $(\frac{1}{\tilde \mu}\tilde \Lambda\odot \tilde \Lambda,\tilde \Lambda,\tilde \mu)$ 
is a solution to \eqref{eq:extsys}. This construction is in general not possible for Lorentzian metrics, 
see Section \ref{sec:Bzeromuzero}.
\end{rem}

\begin{lem}\cite[Lemma 3 and Corollary 5]{KioMatEinstein}\label{lem:projequivEinstein}
Let $(M,g)$ be a connected pseudo-Riemannian Einstein manifold and let $\bar g$ be 
projectively equivalent to $g$ but not affinely equivalent. Then, also $\bar g$ is an Einstein metric.
\end{lem}

Clearly, all projectively equivalent metrics have the same degree of mobility. 
Then, by Lemma \ref{lem:changeofmetric}, Lemma \ref{lem:projequivEinstein} and Proposition \ref{prop:Bnonzero}, 
the degree of mobility of the restriction $g|_{U}$ of $g$ to any open simply connected subset $U$ with compact 
closure is given by one of the values in the list of Theorem \ref{thm:main}. 

The extension ``local $\rightarrow$ global'' follows now directly from \cite[Lemma 12]{FedMat}. 
Alternatively, we may apply \cite[Lemma 10]{MatRos} which is a consequence of the Ambrose-Singer theorem \cite{AmbSing}:
\begin{lem}\cite{MatRos}\label{lem:AmbSing}
Let $\pi:E\rightarrow M$ be a vector bundle with connection $\nabla^E$ over a simply connected $n$-dimensional manifold $M$. Denote by 
$D(E,\nabla^E)$ the dimension of the space of parallel sections and $E|_{U}$ the restriction of $E$ to an open subset $U$ of $M$.

Let $I$ be a subset of integers. Then, if $D(E|_U,\nabla^E)\in I$ for any ball $U$ (that is, $U$ is homeomorphic to a ball 
in $\R^n$ and has compact closure), then also $D(E,\nabla^E)\in I$.
\end{lem}
 
To explain how to apply Lemma \ref{lem:AmbSing} in this situation, it suffices to note that $\A(g)$ is isomorphic
to the space of sections of a certain vector bundle, parallel w.r.t. a certain connection (see \cite[Theorem 3.1]{EastMat}).

In our case the situation is more explicit: $\A(g)$ is isomorphic to the space of solutions of the system 
\eqref{eq:extsys} which can be viewed as the space of sections of the vector bundle $E=S^2T^*M\oplus T^*M\oplus \R$ 
(where the fiber $S^2 T_p^* M$ of $S^2T^*M$ over a point $p\in M$ consists of the symmetric $(0,2)$-tensors on
$T_p M$) which are parallel w.r.t. the connection $\nabla^E$ defined by
$$
\nabla^E_X\left(\begin{array}{c}L\\\Lambda\\\mu\end{array}\right)
=\left(\begin{array}{c}\nabla_X L-X^\flat\odot \Lambda\\
\nabla_X\Lambda-\mu X^\flat-B L(X,.)\\
\nabla_X\mu-2B\Lambda(X)\end{array}\right).
$$

This completes the proof of the first direction of Theorem \ref{thm:main} under the additional assumption 
that $B=0$ in \eqref{eq:extsys} but at least one solution has $\mu\neq 0$.

\subsection{The case when the scalar curvature is zero and all solution of \eqref{eq:extsys} have $\mu=0$}
\label{sec:Bzeromuzero}

The goal of this section is to prove
\begin{prop}\label{prop:Bzeromuzero}
Let $(M,g)$ be a simply connected Ricci flat Lorentzian manifold such that $\mu= 0$ 
for all solutions $(L,\Lambda,\mu)$ of \eqref{eq:extsys} but $\Lambda\neq 0$
for at least one solution. Then, $D(g)$ is given by
$$
D(g)=k(k+1)/2 +l,
$$
where $1\leq k\leq n-4$ and $2\leq l\leq [\frac{n+1-k}{5}]$.
\end{prop}
\begin{rem}\label{rem:Bzeromuzero}
As explained in Remark \ref{rem:changeofmetric}, the case $B=-\mathrm{Scal}/n(n-1)=0$ in \eqref{eq:extsys} and 
$\mu=0$ for all solutions cannot happen if $g$ is Riemannian. This section and Proposition \ref{prop:Bzeromuzero} 
are therefore exclusive for the case of Lorentzian signature.
\end{rem}
We proceed in the same way as in \cite[Section 6.2]{FedMat} and
implement the condition that $g$ is Einstein at the corresponding places.

\begin{lem}\cite[Lemma 13]{FedMat}\label{lem:lambdanull}
Let $(M,g)$ be a simply connected Lorentzian manifold such that all solutions of \eqref{eq:extsys} 
with $B=0$ have $\mu=0$ and at least one solution $(L,\Lambda,0)$ has $\Lambda\neq 0$. 
Then, $\Lambda$ is parallel and orthogonal to any other parallel $1$-form. 
In particular, $|\Lambda|=0$, i.e., $\Lambda$ is a null. 
\end{lem}
Using Lemma \ref{lem:lambdanull}, it is straight-forward to show that any other $\bar L\in \A (g)$ can be written as
$$
\bar L=cL+L'
$$
for a constant $c$ and a parallel symmetric $(0,2)$-tensor $L'$. Thus, 
\begin{align}
D(g)=1+\mathrm{dim}\,\mathrm{Par}^{0,2}(g),\label{eq:degree2}
\end{align}
where $\mathrm{Par}^{0,2}(g)$ denotes the space of parallel symmetric $(0,2)$-tensors for $g$. To find the possible values
of $D(g)$ we therefore have to find the possible values of $\mathrm{dim}\,\mathrm{Par}^{0,2}(g)$. To do so, we use
a maximal holonomy decomposition $TM=\bigoplus_{i=0}^l V_i$ of $TM$ as in \eqref{eq:decomp} into mutually orthogonal holonomy 
invariant subbundles. The difference to the procedure in Section \ref{sec:Bnonzero} is now that $(M,g)$ itself is not a cone 
and also the integral leafs $M_i$ corresponding to the parallel distributions $V_i$ do in 
general not carry the structure of a local cone (although, this is still the case for some components $V_i$ in 
\eqref{eq:decomp} as we shall explain below). We know by Lemma \ref{lem:decomp} that every parallel symmetric $(0,2)$-tensor 
takes the form \eqref{eq:parallel}, hence, 
\begin{align}
\mathrm{dim}\,\mathrm{Par}^{0,2}(g)=\frac{k(k+1)}{2}+l.\label{eq:dimPar}
\end{align}
It remains to determine the range of the integers $k,l$. Since $g$ has Lorentzian signature, precisely one of the metrics
$g_0,...,g_l$ (we use the notation of Lemma \ref{lem:decomp}, that is, $g_i$ is the restriction of $g$ to $V_i$)
has Lorentzian signature. The flat metric $g_0$ is Riemannian, otherwise, by irreducibility of $V_1,...,V_l$, 
the parallel null vector field $\Lambda^\sharp$ must take values in $V_0$. However, since by Lemma \ref{lem:lambdanull}, 
$\Lambda^\sharp$ is orthogonal to any parallel vector field, this implies that $g_0$ is degenerate which is a contradiction. 
Therefore, up to rearranging components, we can suppose that $g_1$ is Lorentzian and $\Lambda^\sharp$ takes values in 
the subbundle $V_1$. It follows that the dimension of the space of parallel vector fields for $g$ is
\begin{align}
k=\mathrm{dim}\,V_0+1.\label{eq:k}
\end{align}
Since $g$ is Ricci flat, each of the components $(M_1,g_1),...,(M_l,g_l)$ 
is Ricci flat ($(M_0,g_0)$ is flat by definition). The next step in \cite{FedMat} is to show that the Riemannian 
manifolds $(M_2,g_2),...,(M_l,g_l)$ each carry the structure of a local cone. Then, since each $(M_i,g_i)$ for $i\geq 2$ is an 
irreducible nonflat Ricci flat local cone, Lemma \ref{lem:dimRicciflat} implies
\begin{align}
\mathrm{dim}\,V_i\geq 5\mbox{ for }i=2,...,l.\label{eq:dimVi}
\end{align}
It remains to establish a lower bound for the dimension of $V_1$. As shown in \cite{FedMat} the restriction $L_1$ 
of $L$ to the manifold $(M_1,g_1)$ is contained in $\A(g_1)$ with 
corresponding $1$-form $\Lambda$ and $(L_1,\Lambda,0)$ satisfies \eqref{eq:extsys} for $g_1$ and constant $B=0$.
Also any other solution to \eqref{eq:extsys} for $g_1$ has $\mu=0$. In \cite[formula (62)]{FedMat} metrics with such properties 
have been described locally. We summarize this description and other facts (see \cite[Lemma 14, 15 and Corollary 2]{FedMat}) in
\begin{lem}\label{lem:localclass}
Let $(N,h)$ be a Lorentzian manifold such that all solutions $(L,\Lambda,\mu)$ of the system \eqref{eq:extsys} 
for $h$ with $B=0$ have $\mu=0$ and let $(L,\Lambda,0)$ be a solution with $\Lambda$ not identically zero. 
Let $\lambda=\frac{1}{2}\mathrm{trace}\,L^\sharp$ such that $\mathrm{grad}\,\lambda$ coincides with the 
parallel null vector field $\Lambda^\sharp$. Then we have the following

\begin{enumerate}
\item The metric $h$ takes the form
\begin{align}
h=h_0+(\lambda+C-\rho_1)^2 h_1+...+(\lambda+C-\rho_m)^2 h_m\label{eq:doublywarped}
\end{align}
in a neighbourhood of almost every point. Here $(N_0,h_0)$ is a $2$-dimensional Lorentzian manifold such 
that $\Lambda$ is contained in $TN_0$, $(N_1,h_1),...,(N_m,h_m)$ are Riemannian manifolds where we have $m\geq 2$, 
and $C$ and $\rho_i$ are certain constants
\item W.r.t. the decomposition $TN=TN_0\oplus...\oplus TN_m$, $L^\sharp$ has block-diagonal form, i.e., $L^\sharp(TN_i)\subseteq TN_i$.
Moreover, $L^\sharp|_{TN_i}=\rho_i\mathrm{Id}_{TN_i}$ for $i=1,...,m$ and $L^\sharp|_{TN_0}$ is conjugate to a $2$-dimensional Jordan block
with eigenvalue $\lambda+C$ and corresponding eigenvector $\Lambda^\sharp$.  
\item If $(N,h)$ is indecomposable, then $\mathrm{dim}\,N_i\geq 2$ for $i=1,...,m$.  
\end{enumerate}
\end{lem}
Using indecomposability of $(M_1,g_1)$, the last statement of the lemma together with $m\geq 2$ shows 
$\mathrm{dim}\,V_1\geq 6$. However, since $g_1$ is Ricci flat, we obtain a sharper lower bound as we will show next.
\begin{lem}\label{lem:curvature}
For $i=0,1,...,m$ let $(N_i,h_i)$ be a pseudo-Riemannian manifold. Consider the product 
$N=N_0\times N_1\times ...\times N_m$ with metric given by
$$
h=h_0+f_1^2 h_2+...+f_m^2 h_m.
$$ 
Suppose the nowhere vanishing functions $f_1,...,f_m$ on $M_0$ are of the form $f_i=\lambda+c_i$ 
for constants $c_i$ and a function $\lambda$ such that $\mathrm{grad}\,\lambda$ is parallel and null. 

Let $R$ and $\mathrm{Ric}$ be the curvature tensor resp. Ricci tensor of $h$. 
Let $X_i,Y_i$ denote vector fields on $N_i$ and let $R^i$ and $\mathrm{Ric}^i$ denote 
the curvature tensor resp. Ricci tensor of $h_i$ for $i=0,1,...,m$. Then, 
\begin{align}
R(X_i,Y_i)=R^i(X_i,Y_i),\,\,\,R(X_j,X_k)=0
\label{eq:curvature}
\end{align}
and
\begin{align}
\mathrm{Ric}(X_i,Y_i)=\mathrm{Ric}^i(X_i,Y_i),\,\,\,\mathrm{Ric}(X_j,X_k)=0
\label{eq:Riccicurvature}
\end{align}
for $i,j,k=0,...,m$, $i\neq j$.
\end{lem}
\begin{proof}
Let $\nabla$ resp. $\nabla^i$ denote the Levi-Civita connection of $h$ resp. $h_i$. 
Using the Koszul formula
$$
2h(\nabla_X Y,Z)=Xh(Y,Z)+Yh(X,Z)-Zh(X,Y)
$$
$$
-h(X,[Y,Z])-h(Y,[X,Z])+h(Z,[X,Y]).
$$
and the expression for $h$, we derive the following formulas, 
relating the Levi-Civita connections $\nabla$ and $\nabla^i$:
\begin{align}
\begin{array}{c}
\nabla_{X_0}Y_0=\nabla^0_{X_0}Y_0,\vspace{1mm}\\
\nabla_{X_0}X_i=\nabla_{X_i}X_0=\frac{\d f_i(X_0)}{f_i}X_i\mbox{ for }i=1,...,m,\vspace{1mm}\\
\nabla_{X_i}Y_i=\nabla^i_{X_i}Y_i-h(X_i,Y_i)\frac{\mathrm{grad}\, f_i}{f_i}\mbox{ for }i=1,...,m,\vspace{1mm}\\
\nabla_{X_i}X_j=0\mbox{ for }i=1,...,m, \,\,\,i\neq j.
\end{array}\label{eq:LC}
\end{align}
Evaluating the curvature tensor $R(X,Y)Z=\nabla_X\nabla_Y Z-\nabla_Y\nabla_X Z-\nabla_{[X,Y]}Z$ 
on the vector fields of various types and using that $h(\mathrm{grad}\,f_i,\mathrm{grad}\,f_j)=|\mathrm{grad}\,\lambda|^2=0$, 
a straight-forward calculation shows that \eqref{eq:curvature} holds and the formulas \eqref{eq:Riccicurvature} follow 
immediately.
\end{proof}
Let us use that the component $(M_1,g_1)$ of $(M,g)$ is Ricci flat. Formula \eqref{eq:Riccicurvature} in 
Lemma \ref{lem:curvature} shows that all components $h_0,h_1,...,h_m$ of $g_1=h$ in \eqref{eq:doublywarped} 
are Ricci flat. Since $3$-dimensional Ricci flat manifolds are flat and, by construction, $g_1$ is nonflat, 
formula \eqref{eq:curvature} shows that at least one of the Ricci flat components $h_i$, $i\geq 1$, 
of $g_1$ in \eqref{eq:doublywarped} is nonflat and therefore must have dimension $\geq 4$. Since there
are at least two components $N_1,N_2$ and $N_0$ is $2$-dimensional, we obtain $\mathrm{dim}\,V_1\geq 8$. 
We claim that this estimate is still too coarse and that instead we actually have
\begin{align}
\mathrm{dim}\,V_1\geq 10.\label{eq:dimV1}
\end{align}
By indecomposability of $(M_1,g_1)$ this follows from
\begin{lem}\label{lem:parallelvf}
Let $(N,h)$ be a simply connected Lorentzian manifold such that all solutions of the system 
\eqref{eq:extsys} for $h$ with $B=0$ have $\mu=0$ and let $(L,\Lambda,0)$ be a solution with $\Lambda$ not 
identically zero. Suppose the metric $h_m$ in the local expression \eqref{eq:doublywarped} from 
Lemma \ref{lem:localclass} is flat.
Let $r$ be the dimension of $N_m$, or equivalently, the multiplicity of the constant eigenvalue $\rho_m$ of $L^\sharp$.

Then, there exist $r$ parallel vector fields $W_1,...,W_r$ on $N$ such that $W_1,...,W_r,\Lambda$ are
linearly independent.  
\end{lem}
Before proving the lemma, we complete the proof of Proposition \ref{prop:Bzeromuzero}. By \eqref{eq:decomp}
and the estimates \eqref{eq:dimVi} and \eqref{eq:dimV1}, we have
$$
n=\mathrm{dim}\,V_0+\mathrm{dim}\,V_1+\mathrm{dim}\,V_2+...+\mathrm{dim}\,V_l\geq \mathrm{dim}\,V_0+5l+5.
$$
Taking into account that $k=\mathrm{dim}\,V_0+1$, this yields $1\leq l\leq [\frac{n+1-k}{5}]-1$. 
From \eqref{eq:degree2} and \eqref{eq:dimPar}, we obtain $D(g)=k(k+1)/2 +l'$ and we have shown that 
$l'=l+1$ is in the range $2\leq l'\leq [\frac{n+1-k}{5}]$. Finally, the estimate \eqref{eq:dimV1}
shows $0\leq \mathrm{dim}\,V_0\leq n-5$, hence, $1\leq k\leq n-4$. This proves Proposition \ref{prop:Bzeromuzero}.
\begin{proof}[Proof of Lemma \ref{lem:parallelvf}]
Actually the statement is a generalization of \cite[Lemma 15(2)]{FedMat} and we will proceed along the same 
line of arguments to give a proof of it. We work in the local picture described by Lemma \ref{lem:localclass} above. 
Let $u$ be a function on $N_m$ such that $\d u$ is parallel and $|\d u|_m=1$, where $|\,.\,|_m$ denotes the length 
of a vector w.r.t. $h_m$. Consider the vector field $U$ on $N$ such that $h(U,X)=u(X)$ for all $X\in TN$. Then,
\begin{align}
h(U,U)=\frac{1}{(\lambda+C-\rho_m)^2}.\label{eq:normU}
\end{align}
Note also that $\nabla U$ is a $h$-symmetric $(1,1)$-tensor on $TN$ and since $U$ takes values in $TN_m$, 
we have $(L^\sharp-\rho_m\Id)(U)=0$ (see Lemma \ref{lem:localclass}(2)). Taking the 
covariant derivative of this equation in the direction of a vector $X\in TN$, inserting \eqref{eq:main} to
replace derivatives of $L^\sharp$ and using $h(\Lambda,U)=0$, we obtain
\begin{align}
(L^\sharp-\rho_m\Id)\nabla_X U=-h(U,X)\Lambda^\sharp.\label{eq:eigvec}
\end{align}
Contracting this with $Y\in TN$ such that $(L^\sharp-\rho\Id)Y=0$ and using symmetries of $\nabla U$, we obtain
$$
(\rho-\rho_m)h(\nabla_Y U,X)=-h(U,X)\Lambda(Y).
$$
Recall from Lemma \ref{lem:localclass}(2) that $L^\sharp(\Lambda^\sharp)=(\lambda+C)\Lambda^\sharp$. Then we have
\begin{align}
\nabla_Y U=0\mbox{ for }Y\in TN_i,\,\,\,i=1,...,m-1,\mbox{ and }Y=\Lambda^\sharp.\label{eq:direc1}
\end{align}
Now let $\tilde \Lambda\in TN_0$ be a vector such that $L(\tilde \Lambda)=(\lambda+C)\tilde \Lambda+\Lambda^\sharp$
(recall that by Lemma \ref{lem:localclass}(2), $L^\sharp|_{TN_0}$ is a Jordan block). Contracting \eqref{eq:eigvec} 
with $\tilde \Lambda$, a straight-forward calculation yields
\begin{align}
\nabla_{\tilde \Lambda} U=-\frac{\Lambda(\tilde\Lambda)}{\lambda+C-\rho_m}U.\label{eq:direc2}
\end{align}

To finally determine $\nabla U$ on a basis of $TN$, let $V$ be another vector tangent to $N_m$. 
Since $U=h^{-1}\d u=\frac{1}{f_m^2}h_m^{-1}\d u$ (where $f_m=\lambda+C-\rho_m$), we have that $f_m^2 U$ 
is a parallel vector field on $N_m$ and using \eqref{eq:LC}, we calculate
$$
2f_m V(f_m)U+f_m^2\nabla_V U=-f_m h(V, U)\Lambda.
$$ 
Hence, since $f_m=\lambda+C-\rho_m$ and $V(f_m)=0$, we obtain
\begin{align}
\nabla_V U=-\frac{1}{\lambda+C-\rho_m}h(V, U)\Lambda^\sharp.\label{eq:direc3}
\end{align}
Now consider the vector field
$$
W=(\lambda+C-\rho_m)U+u\Lambda^\sharp.
$$
By definition, $W$ and $\Lambda^\sharp$ are linearly independent. Using $\Lambda^\sharp=\mathrm{grad}\,\lambda$, $|\Lambda|=0$, 
$\nabla\Lambda^\sharp=0$ and the formulas \eqref{eq:direc1}, \eqref{eq:direc2} and \eqref{eq:direc3}, it is an easy calculation to 
show that the covariant derivative of $W$ vanishes in all possible directions, hence, $W$ is parallel and 
linearly independent of $\Lambda^\sharp$. However, we have defined such a $W$ only in a neighbourhood of almost
every point of $N$. Actually, what we have shown above is the existence of parallel vector fields $W_1,...,W_r$, 
where $r=\mathrm{dim}\,N_m$, defined in a neighbourhood of almost every point, such that $\Lambda^\sharp,W_1,...,W_r$ are 
linearly independent. To see this, we use that $h_m$ is flat and choose a basis of parallel $1$-forms $\d u_1,...,\d u_r$ 
of $N_0$ such that $|\d u_i|_m=1$ for $i=1,...,r$. As shown above, the vector fields 
$W_i=(\lambda+C-\rho_m)U_i+u_i\Lambda^\sharp$, $i=1,...,r$, where $U_i=h^{-1}\d u_i$, will satisfy the claim.

Thus, we have defined a distribution $\tilde D=\mathrm{span}\{\Lambda,W_1,...,W_r\}$ of rank $r+1$ on a dense and open subset 
of $N$. We claim $\tilde D$ extends to a smooth distribution $D$ on the whole $N$. Let $E_i(p)$, $i=1,...,m$, denote the 
generalized eigenspace of $L$ at $p\in N$ corresponding to the constant eigenvalue $\rho_i$. Then we define
$$
D_p=\{X\in T_p N:X\perp E_i,\,\,\,i=1,...,m-1,\,\,\,X\perp \Lambda^\sharp\}
$$
in points $p\in N$ where $(\lambda+C)(p)\neq \rho_i$, $i=1,...,m-1$, and
$$
D_p=\R\cdot \Lambda^\sharp(p)\oplus E_m(p)
$$
for $(\lambda+C)(p)\neq \rho_m$. Then, $D=\bigsqcup_{p\in N}D_p$ is a smooth distribution of rank $r+1$ which
coincides with the parallel and flat distribution $\tilde D$ on a dense and open subset. Then, $D$ is a parallel
and flat subbundle of $TN$. This finishes the proof of the lemma.
\end{proof}

\subsection{Realization of the values of the degree of mobility}
\label{sec:realization}
In this section, we show that for each $n\geq 3$, the values from Theorem \ref{thm:main} can be realized as the 
degree of mobility of an $n$-dimensional Riemannian resp. Lorentzian Einstein metric which admits a projectively
equivalent metric that is not affinely equivalent. This will complete the proof of Theorem \ref{thm:main}. 
We may suppose that $n\geq 5$ since the values of Theorem \ref{thm:main} for $n=3,4$ are realized by the simply 
connected spaces of constant sectional curvature.

\medskip
We will proceed by constructing a Ricci flat local cone $(\hat M,\hat g)$ of suitable signature
and of dimension $n+1$ such that the space of parallel symmetric $(0,2)$-tensors of $\hat g$ has dimension $k(k+1)/2+l$
, where the range of integers $k,l$ is as in Theorem \ref{thm:main}. 
Once such a manifold is constructed, we have by Lemma \ref{lem:curvfacts} and Lemma \ref{lem:localcone}
that $(\hat M,\hat g)$ is (locally) the metric cone over a $n$-dimensional Einstein manifold 
and, in view of Lemma \ref{lem:isom}, the degree of mobility of $(M,g)$ is given by $k(k+1)/2+l$. Moreover, 
as can be seen directly from the second and third equations in \eqref{eq:extsys-1}, any $L\in \mathcal{A}(g)$
that is parallel (that is, we have $\Lambda=0$ for the corresponding vector field) is necessarily proportional 
to the identity. In particular, $(M,g)$ admits a metric projectively equivalent to $g$ and not affinely 
equivalent to it. 

The Ricci flat cone $(\hat M,\hat g)$ will be constructed by taking a direct product of cones. It is therefore
useful to note the following: for any dimension $d+1\geq 5$, there is a Ricci flat nonflat indecomposable
cone of any signature $(r,s+1)$ (where $d=r+s$). By Lemma \ref{lem:curvfacts}, such a cone is obtained by taking the 
metric cone over a generic $d$-dimensional Einstein metric of scalar curvature $d(d-1)$ and signature $(r,s)$.

\medskip 
We will consider two different cases corresponding respectively to the values from the list of 
Theorem \ref{thm:main} attained by Riemannian \emph{and} Lorentzian Einstein metrics and to 
the special values only obtained by Lorentzian Einstein metrics.

\emph{1. Case: Let $0\leq k\leq n-4$ and $1\leq l\leq [\frac{n+1-k}{5}]$}.
Let $M_0=\R^k$ with standard flat euclidean metric $g_0$. Clearly, $(M_0,g_0)$ is a cone over 
the $k-1$-dimensional sphere with standard metric. Since $l\leq [(n+1-k)/5]$, there exist numbers $d_1,...,d_l$ such that 
$d_i\geq 5$ for $i=1,...,l$ and $d_1+...+d_l=n+1-k$. For each $i=1,...,l$, we take $d_i$-dimensional nonflat Ricci flat 
indecomposable cones $(M_i,g_i)$ such that $g_1,...,g_{l-1}$ are positive definite. If we want $g$ to be Riemannian,
we also let $g_l$ be positive definite. If we want $g$ to be Lorentzian, we let $g_l$ be the metric cone over a Lorentzian
Einstein metric. Then, the direct product
$$
(\hat M,\hat g)=(M_0,g_0)\times ( M_1, g_1)\times ...\times ( M_l, g_l)
$$
has Lorentzian signature and the space of parallel symmetric $(0,2)$-tensors 
has dimension $k(k+1)/2+l$. By  Lemma \ref{lem:isom}, Lemma \ref{lem:curvfacts} and Lemma \ref{lem:localcone},
$(\hat M,\hat g)$ is (locally) the metric cone over a $n$-dimensional Einstein manifold $(M,g)$ with 
degree of mobility $D(g)=k(k+1)/2+l$.

\medskip 
\emph{2. Case: Let $2\leq k\leq n-3$, $k=n-3$ mod $5$ and $l= [\frac{n+2-k}{5}]$}.
We let $M_0=\R^{k-2}$ with standard flat euclidean metric $g_0$. 
Since $l-1= [\frac{n-3-k}{5}]$, we find numbers $d_1,...,d_{l-1}\geq 5$ such that $d_1+...+d_{l-1}=n-3-k$.
Let $(M_i,g_i)$, $i=1,...,l-1$, be $d_i$-dimensional nonflat Ricci flat indecomposable cones of Riemannian signature.
Let $(M_l,g_l)$ be the $6$-dimensional cone of signature $(4,2)$ from Example \ref{ex:counterex}.
Consider the $n+1$-dimensional manifold 
$$
(\hat M,\hat g)=(M_0,-g_0)\times ( M_1, -g_1)\times ...\times ( M_{l-1}, -g_{l-1})\times ( M_l, g_l).
$$
of signature $(n-1,2)$. By construction, it has the property that the space of parallel symmetric $(0,2)$-tensors 
has dimension $k(k+1)/2+l$. For $i=0,...,l$ let us write $(M_i,g_i)$ in the form $M_i=\R_{>0}\times N_i$ 
and $g_i=\d r_i^2+r_i^2 h_i$. 
We consider the subset $\hat M^0=\{-r_0^2-r_1^2-...-r_{l-1}^2+r_l^2>0\}\subseteq \hat M$ of points where
the cone vector field $\xi=\sum_{i=0}^l \xi_i$ of $(\hat M,\hat g)$ 
($\xi_i=r_i\partial_{r_i}$ denoting the cone vector fields for $g_i$) 
has the property that $\hat g(\xi,\xi)>0$. As above, we have that, locally, in a neighborhood 
of almost every point of $\hat M^0$, $\hat g$ is the metric cone over an Einstein metric $g$ of 
signature $(n-1,1)$ such that $D(g)=k(k+1)/2+l$.

\section{Proof of Theorem \ref{thm:proj}}
\label{sec:proofproj}

In this section, we give the proof of Theorems \ref{thm:proj} and \ref{thm:proj2}. 
Let $(M,g)$ be an $n$-dimensional pseudo-Riemannian manifold and let $v$ be a projective vector field for $g$. 
It is straight-forward to show that the symmetric $(0,2)$-tensor
\begin{align}
\varphi(v):=\mathcal{L}_v g-\frac{1}{n+1}\mathrm{trace}(\mathcal{L}_v g)^\sharp\label{eq:phiv}
\end{align}
is a solution of \eqref{eq:main}, hence, we have a linear mapping $\varphi:\mathfrak{p}(g)\rightarrow \mathcal{A}(g)$,
where $\mathfrak{p}(g)$ denotes the Lie algebra of projective vector fields.
Using \eqref{eq:phiv}, one easily concludes (see \cite[Lemma 16]{FedMat}) that $\varphi(v)$ is proportional 
to the metric $g$, if and only if $v$ is a homothety (that is, $\mathcal{L}_v g=cg$ for some constant $c$). 
Then, denoting by $\mathfrak{h}(g)$ the Lie algebra of homotheties of $g$, we obtain an induced linear injection 
of quotient spaces
\begin{align}
\varphi:\mathfrak{p}(g)/\mathfrak{h}(g)\rightarrow \mathcal{A}(g)/\R\cdot g,\label{eq:injective}
\end{align}
in particular,
\begin{align}
\mathrm{dim}\left(\mathfrak{p}(g)/\mathfrak{h}(g)\right)\leq D(g)-1.\label{eq:ineq}
\end{align}

Let $g$ be an Einstein metric and assume moreover, that there exists a nonparallel $L\in \A(g)$. 
By Theorem \eqref{thm:extsys}, the degree of mobility $D(g)$ of $g$
equals the dimension of the space of solutions of \eqref{eq:extsys}. As in the proof of 
Theorem \ref{thm:main}, we have to consider different cases according to value of the 
scalar curvature of $g$.

\subsection{The case of nonzero scalar curvature and the realization part of Theorem \ref{thm:proj}} 
\label{sec21}

Let us prove Theorem \ref{thm:proj} under the assumption that the scalar curvature of $g$ is nonzero
(see \cite[Section 8.3]{FedMat} for details): 
using that the constant $B=-\mathrm{Scal}/n(n-1)$ in \eqref{eq:extsys} is nonzero, one shows that any 
homothety for $g$ is actually a Killing vector field, hence, $\mathfrak{h}(g)$ coincides with 
$\mathfrak{i}(g)$, the Lie algebra of Killing vector fields of $g$. Using the equations from the system 
\eqref{eq:extsys} it is straight-forward to show that the injective mapping $\varphi$ in \eqref{eq:injective} 
is actually an isomorphism, hence, $\mathrm{dim}\left(\mathfrak{p}(g)/\mathfrak{i}(g)\right)= D(g)-1$.
Applying Theorem \ref{thm:main} to obtain the values for $D(g)$, we obtain the corresponding values
for the dimension of the space $\mathfrak{p}(g)/\mathfrak{i}(g)$ of essential projective vector fields
from Theorem \ref{thm:proj}. 

The realization part of Theorem \ref{thm:main} also shows that each number from the list of Theorem \ref{thm:proj} 
can actually be realized as the dimension of the space of essential projective vector fields for a certain 
Riemannian resp. Lorentzian Einstein metric. This proves the realization part of Theorem \ref{thm:proj}. 

Let us turn to the prove of Theorem \ref{thm:proj2} in case of nonzero scalar curvature. Let $g$
be an Einstein metric of arbitrary signature and with nonzero scalar curvature which admits a 
projectively equivalent metric that is not affinely equivalent. Let $(L,\Lambda,\mu)$ be a solution 
of \eqref{eq:extsys} such that $\Lambda\neq 0$. 
It is wellknown that for $B\neq 0$, $\Lambda^\sharp$ is an essential projective vector field for $g$
which proves Theorem \ref{thm:proj2}. For completeness let us show how to verify this fact:
we have
$$
\mathcal{L}_{\Lambda^\sharp}g=2\nabla\Lambda=2\mu g+2BL,
$$
hence,
$$
\mathrm{trace}(\mathcal{L}_{\Lambda^\sharp}g)^\sharp=2n\mu +4B\lambda,
$$
where $\lambda=\frac{1}{2}\mathrm{trace}\,L^\sharp$. Since $\d\lambda=\Lambda$ and $\d \mu=2B\Lambda$,
we have that $\mu-2B\lambda$ is equal to a constant. Using this, we obtain
\begin{align}
\mathcal{L}_{\Lambda^\sharp} g-\frac{1}{n+1}\mathrm{trace}(\mathcal{L}_{\Lambda^\sharp} g)^\sharp
=2BL-C g\in \A(g),\label{eq:splitting}
\end{align}
where $C$ is a certain constant. This shows that $\Lambda^\sharp$ is an essential projective vector field
(compare \eqref{eq:phiv}) and proves Theorem \ref{thm:proj2} for nonzero scalar curvature. 

\begin{rem}\label{rem:essprojvf}
We see from \eqref{eq:splitting} that the mapping 
$$
s:\A(g)/\R\cdot g\rightarrow \mathfrak{p}(g)/\mathfrak{i}(g),
$$
defined by sending $L\in \A(g)$ to the corresponding vector field $\frac{1}{2B}\Lambda^\sharp$, 
is a splitting of the exact sequence 
$$
0\rightarrow \mathfrak{i}(g)\hookrightarrow \mathfrak{p}(g)\overset{\varphi}{\rightarrow }\A(g)/\R\cdot g,
$$
that is $\varphi\circ s=\Id$. In particular, the space of essential projective vector fields 
$\mathfrak{p}(g)/\mathfrak{i}(g)$ can be identified with a subspace of $\mathfrak{p}(g)$
(which is not a subalgebra) and each projective vector field for $g$ is of the form  
$\Lambda+K$, where $K$ is a Killing vector field.
\end{rem}

\subsection{The case of zero scalar curvature and $\mu\neq 0$ for at least one solution of \eqref{eq:extsys}}

The proof of Theorem \ref{thm:proj} under the assumption that $B=-\mathrm{Scal}/n(n-1)=0$ in the system \eqref{eq:extsys} 
and at least one solution has $\mu\neq 0$ can be traced back to the case $B\neq 0$ treated in the  
previous section. We first recall some invariance properties.
\begin{lem}\label{lem:invariance}
We have 
$\mathrm{dim}(\mathfrak{p}(g)/\mathfrak{i}(g))=\mathrm{dim}(\mathfrak{p}(\bar g)/\mathfrak{i}(\bar g))$
for any pair of projectively equivalent metrics $g,\bar g$.
\end{lem}
\begin{proof}
By definition of a projective vector field, we have $\mathrm{dim}(\mathfrak{p}(g))=\mathrm{dim}(\mathfrak{p}(\bar g))$
On the other hand, since the defining equation for a Killing vector field is projectively invariant
(when we view it as an equation on weighted $1$-forms, see \cite{East}), 
we also have $\mathrm{dim}(\mathfrak{i}(g)=\mathrm{dim}(\mathfrak{i}(\bar g))$ and the claim follows.
\end{proof}

By Lemma \ref{lem:changeofmetric}, on each open simply connected subset $U$ of $M$
with compact closure, there exists a metric $\bar g$ having the same signature as $g$ and being
projectively equivalent to $g$ such that $\bar B\neq 0$ for the corresponding constant in the system \eqref{eq:extsys} 
for $\bar g$. By Lemma \ref{lem:projequivEinstein}, also $\bar g$ is an Einstein metric. It follows from 
Lemma \ref{lem:invariance} and the results of Section \ref{sec21} that for each simply connected open 
subset $U$ with compact closure, $\mathrm{dim}(\mathfrak{p}(g|_U)/\mathfrak{i}(g|_U))$ is given by one of the 
values from the list of Theorem \ref{thm:proj}. 
However, it is a classical fact that Killing vector fields can be viewed equivalently 
as parallel sections on a certain vector bundle. The same is true for the projective vector fields of $g$ 
(since they are the symmetries of the projective geometry determined by the Levi-Civita connection of $g$ 
\cite{CapMel,CapMel2,Melnick} and general facts about parabolic (projective) geometries assure the existence 
of a prolongation connection \cite{Hammerl}). 
Then, the proof of Theorem \ref{thm:proj} under the assumptions $B=0$ but $\mu\neq 0$ for at least one solution
of \eqref{eq:extsys} follows from a standard application of the Ambrose-Singer theorem \cite{AmbSing}, 
see also Lemma \ref{lem:AmbSing} and its proof in \cite[Lemma 10]{MatRos}.

In the same way one proves Theorem \ref{thm:proj2} for an Einstein metric of arbitrary signature with vanishing 
scalar curvature which admits a solution $(L,\Lambda,\mu)$ of \eqref{eq:extsys} such that $\mu\neq 0$: arguing as 
above (using Lemma \ref{lem:changeofmetric} and Lemma \ref{lem:projequivEinstein}), the already proven part of 
Theorem \ref{thm:proj2} for nonzero scalar curvature (see Section \ref{sec21})
implies that the restriction $g|_U$ of $g$ to any open simply connected subset $U$ with compact closure has
$\mathrm{dim}(\mathfrak{p}(g|_U)/\mathfrak{i}(g|_U))\geq 1$, hence, admits an essential projective vector field.
A standard application of the Ambrose-Singer theorem yields the desired result for $g$.

\subsection{The case of zero scalar curvature and $\mu= 0$ for all solutions of \eqref{eq:extsys}}

Let $(M,g)$ be a simply connected Lorentzian manifold such that every solution of 
the system \eqref{eq:extsys} with $B=0$ has $\mu=0$ and $\Lambda\neq 0$ for at least one solution 
(recall from Remark \ref{rem:Bzeromuzero} that the situation under consideration is exclusive for Lorentzian signature). 
By \cite[Corollary 3]{FedMat}, we have that $\mathfrak{p}(g)=\mathfrak{i}(g)$. Thus, 
$\mathrm{dim}\left(\mathfrak{p}(g)/\mathfrak{i}(g)\right)\leq D(g)-1$
by \eqref{eq:ineq}. It is shown in \cite[Section 8.4.2]{FedMat} that we also have 
$D(g)-2\leq \mathrm{dim}\left(\mathfrak{p}(g)/\mathfrak{i}(g)\right)$, hence
$$
D(g)-2\leq \mathrm{dim}\left(\mathfrak{p}(g)/\mathfrak{i}(g)\right)\leq D(g)-1.
$$
Using Proposition \ref{prop:Bzeromuzero}, we obtain 
$$
\frac{k(k+1)}{2}+l'-2\leq \mathrm{dim}\left(\mathfrak{p}(g)/\mathfrak{i}(g)\right)\leq \frac{k(k+1)}{2}+l'-1,
$$
where $1\leq k\leq n-4$ and $2\leq l'\leq [\frac{n+1-k}{5}]$. Thus, 
$\mathrm{dim}\left(\mathfrak{p}(g)/\mathfrak{i}(g)\right)=k(k+1)/2+l-1$,where $l=l'$ or $l=l'-1$. 
Then, $\mathrm{dim}\left(\mathfrak{p}(g)/\mathfrak{i}(g)\right)=k(k+1)/2+l-1$,
where $1\leq l\leq [\frac{n+1-k}{5}]$. This proves Theorem \ref{thm:proj} under the assumptions $B=0$
and $\mu=0$ for all solutions of \eqref{eq:extsys}.

Finally, let us prove Theorem \ref{thm:proj2} for an Einstein metric of arbitrary signature with vanishing 
scalar curvature such that $\mu=0$ for every solution of \eqref{eq:extsys} but $\Lambda\neq 0$ for at least
one solution $(L,\lambda,0)$. Let $\lambda=\frac{1}{2}\mathrm{trace}(L^\sharp)$ such that $\d\lambda=\Lambda$.
Then, since $\Lambda$ is parallel, $\nabla v^\flat=\Lambda\otimes \Lambda$ for the vector field $v=\lambda\Lambda^\sharp$, 
hence,
$$
\mathcal{L}_v g-\frac{1}{n+1}\mathrm{trace}(\mathcal{L}_v g)^\sharp g=2\Lambda\otimes \Lambda-\frac{2g(\Lambda,\Lambda)}{n+1} g.
$$
Since $g(\Lambda,\Lambda)$ is a constant, this symmetric $(0,2)$-tensor is clearly contained in $\A(g)$. It follows that
$v$ is a projective vector field. Moreover, $v$ is essential since it is not an isometry (thought, $v$ is an affine vector field).

\subsection*{\bf Acknowledgements.} We thank Deutsche Forschungsgemeinschaft (Research training group   1523 --- 
Quantum and Gravitational Fields) and FSU Jena for partial financial support.

\nocite{*}
\bibliographystyle{plain}

\end{document}